\pdfoutput=1
\documentclass[a4paper]{scrartcl}

\usepackage[utf8]{inputenc}
\usepackage[english]{babel}
\usepackage{bm}
\usepackage{amsmath,amsfonts,amssymb,amsthm}
\usepackage{enumitem}
\usepackage{float}
\usepackage{mathtools}
\usepackage{graphicx}
\usepackage{hyperref}
\usepackage{subcaption}
\usepackage{stmaryrd}
\usepackage{tikz}

\newtheorem{theorem}{Theorem}[section]
\newtheorem{lemma}[theorem]{Lemma}

\theoremstyle{definition}

\newtheorem{definition}[theorem]{Definition}
\newtheorem*{remark}{Remark}

\newcommand{\ie}{i.\,e.,~}
\newcommand{\eg}{e.\,g.,~}
\newcommand{\apos}[1]{\grqq{}{#1}"}

\renewcommand{\|}{\,|\,}
\renewcommand{\vec}[1]{\bm{#1}}
\renewcommand{\u}{{\vec{u}}}
\renewcommand{\v}{{\vec{v}}}
\newcommand{\w}{{\vec{w}}}
\newcommand{\x}{{\vec{x}}}
\newcommand{\y}{{\vec{y}}}
\newcommand{\z}{{\vec{z}}}

\newcommand{\V}{{\vec{V}}}

\newcommand{\X}{{\vec{X}}}
\newcommand{\Y}{{\vec{Y}}}
\newcommand{\Z}{{\vec{Z}}}
\renewcommand{\d}[1]{\,\textnormal{d}{#1}}
\newcommand{\dx}{\,\textnormal{d}\vec{x}}
\newcommand{\dy}{\,\textnormal{d}\vec{y}}
\newcommand{\dz}{\,\textnormal{d}\vec{z}}
\newcommand{\Xset}{\mathcal{X}}
\newcommand{\Yset}{\mathcal{Y}}
\newcommand{\Zset}{\mathcal{Z}}
\newcommand{\tr}{^\top}
\newcommand{\trace}[1]{\textnormal{trace}\left({#1}\right)}
\newcommand{\defas}{\coloneqq}
\newcommand{\asdef}{\eqqcolon}

\newcommand{\norm}[1]{\lVert{#1}\rVert}
\newcommand{\abs}[1]{\lvert{#1}\rvert}
\newcommand{\set}[2]{\lbrace {#1}\,|\,{#2} \rbrace}
\newcommand{\grad}{\nabla}
\newcommand{\eps}{\varepsilon}
\newcommand{\prob}{{\mathbf{P}}}
\newcommand{\expct}{{\mathbf{E}}}
\newcommand{\expctcond}[2]{\expct[{#1}\|{#2}]}
\newcommand{\var}{{\mathbf{V}\textnormal{ar}}}
\newcommand{\varcond}[2]{\var({#1}\|{#2})}
\newcommand{\ind}{\mathbf{1}}
\newcommand{\N}{\mathbf{N}}
\newcommand{\R}{\mathbf{R}}
\newcommand{\xyzW}[2]{\llbracket{#1}\rrbracket_{#2}}
\newcommand{\xyz}[1]{\xyzW{#1}{}}
\newcommand{\px}{\rho_\X}
\newcommand{\pyz}{\rho_{\Y,\Z}}
\newcommand{\py}{\rho_\Y}

\newcommand{\pzy}{\rho_{\Z|\Y}}
\newcommand{\expctpzy}[1]{\expctcond{#1}{\Y}}

\newcommand{\inter}[1]{#1^\circ}

\newcommand{\wrt}{{w.r.t.}}
\newcommand{\dom}[1]{\textnormal{dom}({#1})}
\newcommand{\Cpoinc}{{C_\textnormal{P}}}

\newcommand{\diam}[1]{\textnormal{diam}({#1})}
\newcommand{\ytl}{{\check{\y}}}
\DeclareMathOperator{\esssup}{ess\,sup}

\newcommand{\CpoincW}{{C_\textnormal{P,$W$}}}
\newcommand{\Cpoinceps}{{C_\textnormal{P,$\eps$}}}
\newcommand{\CpoincepsW}{{C_\textnormal{P,$\eps$,$W$}}}
\newcommand{\Cvar}{{C_{\textnormal{$\var$}}}}
\newcommand{\CvarW}{{C_{\textnormal{$\var$,$W$}}}}

\newcommand{\wm}{\textit{Wolfram Mathematica}}

\title{Generalized bounds for active subspaces\thanks{{\bfseries\sffamily Funding:} MTP acknowledges support from the \textit{Deutsche Forschungsgemeinschaft} (DFG) through \textit{TUM International Graduate School for Science and Engineering} (IGSSE), GSC~81.
		Additional financial support for BW was provided by DFG through WO~671/11-1.
		JW was supported in part by the Swedish Research Council under grant No. 2018-01726.}}
\author{Mario Teixeira Parente\footnote{Chair for Numerical Mathematics, Technical University of Munich, Boltzmannstraße 3, 85748 Garching bei München, Germany (\texttt{\{parente,wohlmuth\}@ma.tum.de})} \and Jonas Wallin\footnote{Department of Statistics, Lund University, Lund, Sweden (\texttt{jonas.wallin@stat.lu.se})} \and Barbara Wohlmuth$^\dagger$}
\date{\today}

\ifpdf
\hypersetup{
	pdftitle={Generalized bounds for active subspaces},
	pdfauthor={Mario Teixeira Parente, Jonas Wallin, Barbara Wohlmuth}
}
\fi

\graphicspath{{media/}}
\numberwithin{equation}{section}

\allowdisplaybreaks

\begin{document}
\maketitle

\begin{abstract}
In this article, we consider scenarios in which traditional estimates for the active subspace method based on probabilistic Poincaré inequalities are not valid due to unbounded Poincaré constants.
Consequently, we propose a framework that allows to derive generalized estimates in the sense that it enables to control the trade-off between the size of the Poincaré constant and a weaker order of the final error bound.
In particular, we investigate independently exponentially distributed random variables in dimension two or larger and give explicit expressions for corresponding Poincaré constants showing their dependence on the dimension of the problem.
Finally, we suggest possibilities for future work that aim for extending the class of distributions applicable to the active subspace method as we regard this as an opportunity to enlarge its usability.

\end{abstract}

\section{Introduction}
\label{sec:intro}
Many modern computational problems, having a large number of input variables or parameters, suffer from the \apos{curse of dimensionality} in that their solution becomes computationally expensive or even intractable as the dimension of the problem grows.
The \textit{active subspace method} (ASM), or shorter, \textit{active subspaces}~\cite{constantine2015active,constantine2014active}, is a set of tools for dimension reduction which reduce the effects caused by the curse of dimensionality.
ASM splits an Euclidean input space into a so-called active and inactive subspace based on average sensitivities of a real-valued function of interest.
The sensitivities are found by an eigendecomposition of weighted outer products of the function's gradient with itself.
That is, eigenvalues indicate average sensitivities of a function of interest in the direction of the corresponding eigenvector.
Eigenvectors and eigenvalues belonging to the active subspace are then considered as dominant for the global behavior of the function of interest, whereas the inactive subspace is regarded as negligible.

The usefulness of ASM has already been demonstrated for several real case studies in various applied disciplines; see, \eg~\cite{diaz2018modified,leon2017identifiability,teixeiraparente2019efficient,teixeiraparente2019bayesian,tezzele2018combined}.
It has also motivated other methodological advances, \eg in the solution of Bayesian inverse problems~\cite{stuart2010inverse} by an accelerated Markov chain Monte Carlo algorithm~\cite{constantine2016accelerating}, in uncertainty quantification and propagation~\cite{coleman2019gradient,tripathy2016gaussian}, and in the theory of ridge approximation; see, \eg~\cite{constantine2017near,glaws2018lanczos,hokanson2018data}.

However, ASM is only one dimension reduction technique among others.
For example, likelihood-informed dimension reduction for the solution of Bayesian inverse problems~\cite{cui2014likelihood} is based on a similar idea.
This approach, however, analyzes the Hessian matrix of the function of interest instead of the gradient.
An extension to vector-valued functions in gradient-based dimension reduction is given by~\cite{zahm2018gradient}.
Dimension reduction for nonlinear Bayesian inverse problems based on the Kullback-Leibler (KL) divergence of approximate posteriors and (subspace) logarithmic Sobolev inequalities, including a comprehensive comparison of several other techniques, was provided by the authors of~\cite{zahm2018certified}.
Furthermore, Active Manifolds~\cite{bridges2019active}, as a nonlinear analogue to ASM, and PTU~\cite{budninskiy2019parallel}, as an extension to a framework for nonlinear dimension reduction called Isomap~\cite{tenenbaum2000global}, both have demonstrated a lot of promise.

A main result in ASM theory is an upper bound on the mean squared error between the original function of interest and its low-dimensional approximation on the active subspace.
The corresponding proof is based on an inequality of Poincaré type which is probabilistic in nature since ASM involves a probability distribution that weights sensitivities of the function of interest at different locations in the input space.
The upper bound consists of the product of a Poincaré type constant and the sum of eigenvalues corresponding to the inactive subspace, called inactive trace in the following.
The constant derived in~\cite{constantine2014active} is claimed to depend only on the original distribution which is generally incorrect.
Also, to the knowledge of the authors, existing theory for dimension reduction techniques based on Poincaré or logarithmic Sobolev inequalities are subject to quite restrictive assumptions on the involved probability distribution.
These assumptions comprise either the distribution having compact support or its density $\rho$ being of uniformly log-concave form, \ie $\rho(\x)=\exp(-V(\x))$, where $V$ is such that its Hessian matrix $V''(\x)\succeq\alpha I$ for each~$\x$ and some~$\alpha>0$.
By the famous Bakry-Émery criterion, the latter assumption implies a logarithmic Sobolev inequality and Poincaré inequality with universal Poincaré constant~$1/\alpha$; see, \eg~\cite{bakry2013analysis,handel2016probability}.
Note that the case $\alpha=0$, \ie $V$ being only convex, is not covered.
However, Bobkov~\cite{bobkov1999isoperimetric} showed that a Poincaré inequality is still satisfied in this case and gave lower and upper bounds on the corresponding Poincaré constant.
Distributions with heavier tails, \ie for~$\alpha=0$, as, \eg exponential or Laplace distributions, do not satisfy the assumptions above, but are, however, of practical  relevance.

In ASM theory, it is not the original distribution that must satisfy a Poincaré inequality, but a conditional distribution on the inactive subspace, which depends on a variable defined on the active subspace.
Both assumptions on the original distribution from above are in fact passed on to the conditional distribution.
However, the case $\alpha=0$ is cumbersome.
We shall give an example for this case regarding a distribution that itself satisfies a Poincaré inequality, but might not be applicable at all or only with care due to an arbitrary large Poincaré constant in the final bound for the mentioned mean squared error.
Our arguments are based on the bounds for corresponding Poincaré constants given by Bobkov in~\cite{bobkov1999isoperimetric}.
We also describe a way to still get upper bounds in this situation, however with a weaker, reduced order in the inactive trace.
This order reduction is controllable in the sense that the practitioner can decide for the actual trade-off between the order of the inactive trace and the size of the corresponding Poincaré constant.
The mentioned general problem and its solution is exemplified on independently exponentially distributed random variables in dimension two and larger.
Also, it is shown that the final constant is very much depending on the dimension of the problem.
However, since this example is rather special, we eventually propose opportunities for future work that aim for extending the class of distributions for which the bounds and the involved constants are explicitly available in order to expand the applicability of ASM to more scenarios of practical interest.
In particular, the class of multivariate generalized hyperbolic distributions is a rich class that is, in our opinion, worthwhile to get investigated.
Details on arising difficulties with this class are also provided.

The outline of the manuscript is as follows.
Section~\ref{sec:act_subsp} gives an introduction to ASM and its formal context.
In Section~\ref{sec:comp_norm_dists}, we recall results involving compactly supported and normal distributions.
The main results consisting of a motivation and discussion of the mentioned problems, with independently exponentially distributed random variables as an extreme example, are presented in Section~\ref{sec:main_results}.
In Section~\ref{sec:futu_work}, we propose possibilities for future work.
Finally, a summary is given in Section~\ref{sec:summ}.

\section{Active subspaces}
\label{sec:act_subsp}
The active subspace method is a set of tools for gradient-based dimension reduction~\cite{constantine2015active,constantine2014active}.
Its aim is to find directions in the domain of a function $f$ along which the function changes dominantly, on average.
For illustration, consider a function of the form $f(\x) = g(A\tr\x)$ with a so-called profile function $g$ and a matrix $A\in\R^{n\times k}$, $1\leq k\leq n$, $n\geq2$.
Functions of this type are called \textit{ridge functions}~\cite{pinkus2015ridge}.
Note that $f$ is constant along the null space of $A\tr$.
Indeed, for $\x\in\dom{f}\subseteq\R^n$ and $\v\in\mathcal{N}(A\tr)$ such that $\x+\v\in\dom{f}$, it holds that
\begin{equation}
	f(\x+\v) = g(A\tr(\x+\v)) = g(A\tr\x) = f(\x).
\end{equation}
That is, $f$ is intrinsically at most $k$-dimensional.
For arbitrary~$f$, the general task is to find a suitable dimension~$k$, a function $g:\dom{g}\to\R$, $\dom{g}\subseteq\R^k$, and a matrix $A\in\R^{n\times k}$ such that $f(\x)\approx g(A\tr\x)$.

For this, the active subspace method assumes that the function of interest $f:\Xset\to\R$ is continuously differentiable with partial derivatives that are square-integrable \wrt{} a probability density function~$\px$.
We define~$\Xset\defas\dom{f}\subseteq\R^n$ to be the support of~$\px$, \ie the closure of the set $\Xset^+\defas\set{\x\in\R^n}{\px(\x)>0}$.
We assume that $\Xset$ is a continuity set, that is, its boundary is a Lebesgue null set.
The central object of interest is a matrix constructed by outer products of the gradient of~$f$, $\grad f=\grad^\x f$, with itself weighted by~$\px$,
\begin{align}
	\label{eq:C}
	C &\defas \int_{\R^n}{\grad f(\x) \, \grad f(\x)\tr \px(\x) \dx}.
\end{align}
Since $C$ is real symmetric, there exists an eigendecomposition $C=W\Lambda W\tr$ with an orthogonal matrix~$W\in\R^{n\times n}$ and a diagonal matrix~$\Lambda\in\R^{n\times n}$ with descending eigenvalues $\lambda_1,\ldots,\lambda_n$ on its diagonal.
The positive semidefiniteness of $C$ additionally ensures that $\lambda_1\geq\cdots\geq\lambda_n\geq0$.
Note that the matrices~$C$, $W$, and~$\Lambda$ all depend on~$f$ and~$\px$.

The behavior of the function $f$ and the eigendecomposition of $C$ have an interesting, exploitable relation, \ie
\begin{equation}
	\lambda_i = \w_i\tr C\w_i = \int_{\R^n}{(\w_i\tr\grad f(\x))^2\px(\x)\dx}, \quad i=1,\ldots,n.
\end{equation}
If, for example, $\lambda_i=0$ for some $i$, then we can conclude that $f$ does not change in the direction of the corresponding eigenvector $\w_i$.
That is, if eigenvalues $\lambda_i$, $i=k+1,\ldots,n$, are sufficiently small for a suitable~$k\leq n-1$, or even zero as in the case of ridge functions, then $f$ can be approximated by a lower-dimensional function.
Formally, this corresponds to a split of $\Lambda$ and $W$, \ie
\begin{equation}
	\Lambda  = \begin{pmatrix}\Lambda_1 & \\ & \Lambda_2\end{pmatrix} \quad\text{and}\quad W = \begin{pmatrix}W_1 & W_2\end{pmatrix},
\end{equation}
where $\Lambda_1\in\R^{k\times k}$, $\Lambda_2\in\R^{n-k\times n-k}$ and $W_1\in\R^{n\times k}$, $W_2\in\R^{n\times n-k}$.

Since
\begin{align}
	\x = WW\tr\x = W_1W_1\tr\x+W_2W_2\tr\x = W_1\y+W_2\z,
\end{align}
the split of $W$ suggests a new coordinate system $(\y,\z)$ for the \textit{active variable} $\y\defas W_1\tr\x\in\R^k$ and the \textit{inactive variable} $\z\defas W_2\tr\x\in\R^{n-k}$.
The range of $W_1$, $\mathcal{R}(W_1)\defas\set{W_1\y}{\y\in\R^k}\subseteq\R^n$, is called the \textit{active subspace of $f$}.
Note that the new variable $\y$ is aligned to directions on which $f$ changes much more, on average, than on directions the variable $\z$ is aligned to.

For the remainder, we define
\begin{equation}
	\label{eq:Yset_Zset}
	\Yset\defas W_1\tr\Xset=\set{W_1\tr\x}{\x\in\Xset} \quad\text{and}\quad \Zset\defas W_2\tr\Xset=\set{W_2\tr\x}{\x\in\Xset}.
\end{equation}
Also, for $\y\in\Yset$ and $\z\in\Zset$, let
\begin{equation}
	\xyz{\y,\z} \defas \xyzW{\y,\z}{W} \defas W_1\y+W_2\z
\end{equation}
to concisely denote changes of the coordinate system.

Variables $\x$, $\y$, and $\z$ can also be regarded as random variables $\X$, $\Y$, and $\Z$, respectively, that are defined on a common probability space~$(\Omega,\mathcal{F},\prob)$.
The orthogonal variable transformation $\x\mapsto(\y,\z)$ induces probability density functions for the random variables~$\Y$ and~$\Z$.
That is, the joint distribution of~$(\Y,\Z)$ is
\begin{equation}
	\pyz(\y,\z)=\px(\xyz{\y,\z})
\end{equation}
for~$\y\in\Yset$ and~$\z\in\Zset$.
Corresponding marginal and conditional densities are defined as usual.
Additionally, set
\begin{equation}
	\label{eq:Yset_pl}
	\Yset^+\defas\set{\y\in\R^k}{\py(\y)>0}
\end{equation}
to denote the set of all values for the active variable $\y$ with a strictly positive density value.
We frequently use that for a $\px$-integrable function $h:\Xset\to\R$, it holds that
\begin{equation}
	\expct[h(\X)] = \expct[\expctpzy{h(\xyz{\Y,\Z})}].
\end{equation}

Given the eigenvectors in $W$, we still need to define a lower-dimensional function~$g$ approximating~$f$.
For $\y\in\Yset^+$, a natural way is to define~$g(\y)$ as the conditional expectation of~$f$ given~$\y$, \ie as an integral over the inactive subspace weighted with the conditional density~$\pzy(\cdot|\y)$.
Recall that this approximation is the best in an $L^2$ sense \cite[Corollary~8.17]{klenke2013probability}.
Hence, we set
\begin{equation}
	\label{eq:g}
	\begin{aligned}
		g(\y) \defas{}& \expct[f(\xyz{\Y,\Z})\|\Y=\y] \\
		={}& \int_{\R^{n-k}}{f(\xyz{\y,\z})\,\pzy(\z|\y)\dz}.
	\end{aligned}
\end{equation}
Additionally, we define
\begin{equation}
	f_g(\x)\defas g(W_1\tr\x)
\end{equation}
for $\x\in\inter{\Xset}$, where $\inter{\Xset}$ denotes the interior of $\Xset$.
Note that $W_1\tr\x\in\inter{\Yset}\subseteq\Yset^+$ for $\x\in\inter{\Xset}$.

\begin{remark}
	In practice, both the matrix~$C$ from~\eqref{eq:C} and the low-dimensional function~$g$ from~\eqref{eq:g} are often not exactly available.
	Our results can, however, also adapted to a corresponding perturbation analysis, provided in~\cite{constantine2014active}, which we do not perform since it would require additional notation and complexity but not contribute to the central aspects of this manuscript.
\end{remark}

One of the main results in ASM theory is a theorem that gives an upper bound on the mean squared error of~$f_g$ approximating~$f$.
The upper bound is the product of a Poincaré constant~$\CpoincW>0$ and the sum of~$n-k$ eigenvalues corresponding to the inactive subspace, called inactive trace.
That is, if the inactive trace is small, then the mean squared error of~$f_g$ approximating~$f$ is also small.
Mathematically, for a given probability density function $\px$, the theorem states that \cite[Theorem~3.1]{constantine2014active}
\begin{equation}
	\label{eq:act_subsp_ineq_poincW}
	\expct[(f(\X)-f_g(\X))^2] \leq \CpoincW (\lambda_{k+1}+\cdots+\lambda_n)
\end{equation}
for a Poincaré constant $\CpoincW=\CpoincW(W,\px)>0$.
Note that~$\CpoincW$ depends on~$W=W(f)$ and thus also indirectly on~$f$.
If desired, we could remove this dependence by considering the supremum of~$\CpoincW$ over all orthogonal matrices, \ie
\begin{equation}
	\label{eq:C_poinc}
	\Cpoinc \defas \sup_{W \text{orth.}}\CpoincW,
\end{equation}
and get
\begin{equation}
	\label{eq:act_subsp_ineq_poinc}
	\expct[(f(\X)-f_g(\X))^2] \leq \Cpoinc (\lambda_{k+1}+\cdots+\lambda_n),
\end{equation}
provided the constant~$\Cpoinc=\Cpoinc(\px)$ exists.
Deriving such an upper bound for a certain class of distributions would allow to choose~$\px$ independently of~$f$.
Note that~\cite{zahm2018gradient,zahm2018certified} also control the Poincaré constant for any orthogonal matrix~$W$.

The derivation of~\eqref{eq:act_subsp_ineq_poincW} starts with
\begin{align}
	\expct[(f(\X)-f_g(\X))^2] &= \expct[\expctpzy{(f(\xyz{\Y,\Z})-g(\Y))^2}] \label{eq:mse_f_fg_fX_fYZ} \\
	&\leq \expct[C_\Y \, \expctpzy{\norm{\grad^\z f(\xyz{\Y,\Z})}_2^2}], \label{eq:poinc_pzy}
\end{align}
where we used a probabilistic Poincaré inequality \wrt{} $\pzy(\cdot|\y)$ for a given $\y\in\Yset^+$.
Note that the Poincaré constant~$C_\y$ of $\pzy(\cdot|\y)$ depends on~$\y$.
In~\cite[Theorem~3.1]{constantine2014active}, it was indirectly assumed that this constant does not depend on~$\y$.
Under the assumption that $\CpoincW\defas\esssup{C_\Y}<\infty$, \ie the distribution of~$C_\Y$ has compact support, we can continue with
\begin{equation}
	\label{eq:supCy_exists}
	\expct[(f(\X)-f_g(\X))^2] \leq \CpoincW \, \expct[\expctpzy{\norm{\grad^\z f(\xyz{\Y,\Z})}_2^2}].
\end{equation}
However, as we see in Subsection~\ref{ssec:exp_dist_case}, this assumption on~$C_\Y$ is not always fulfilled.

The continuation of~\eqref{eq:supCy_exists} follows~\cite[Lemma~2.2 and Theorem~3.1]{constantine2014active}.
We repeat the steps here for the sake of completeness.
So, first, note that $\grad^\z f(\xyz{\y,\z}) = W_2\tr\grad^\x f(\xyz{\y,\z})$ for $\y\in\Yset$ and~$\z\in\Zset$.
Then, we write
\begin{align}
	\label{eq:grad_eigvals}
	\begin{split}
		&\expct[\expctpzy{\norm{\grad^\z f(\xyz{\Y,\Z})}_2^2}] \\
		&\qquad= \trace{\expct[\expctpzy{\grad^\z f(\xyz{\Y,\Z})\,\grad^\z f(\xyz{\Y,\Z})\tr}]} \\
		&\qquad= \trace{W_2\tr\,\expct[\expctpzy{\grad^\x f(\xyz{\Y,\Z})\,\grad^\x f(\xyz{\Y,\Z})\tr}]\,W_2} \\
		&\qquad= \trace{W_2\tr\,\expct[\grad^\x f(\X)\,\grad^\x f(\X)\tr]\,W_2} \\
		&\qquad= \trace{W_2\tr CW_2} = \trace{W_2\tr W\Lambda W\tr W_2} \\
		&\qquad= \trace{\Lambda_2} = \lambda_{k+1}+\cdots+\lambda_n.
	\end{split}
\end{align}

The next section gives two examples for types of densities $\px$ that are well-known to imply a probabilistic Poincaré inequality for $\pzy(\cdot|\y)$ and allow for a bound on its constant~$C_\y$ that is uniform in~$\y$ and~$W$.
Again, we emphasize that it is not~$\px$ that should satisfy a probabilistic Poincaré inequality but~$\pzy(\cdot|\y)$.

\section{Compactly supported and normal distributions}
\label{sec:comp_norm_dists}
The uniform distribution, as a canonical example of a distribution with compact support~$\Xset$, is well-known to satisfy a probabilistic Poincaré inequality on its own and to imply the same for densities~$\pzy(\cdot|\y)$ which are also uniform.
Note that a probabilistic Poincaré inequality involving a uniform distribution is actually equivalent to a regular Poincaré inequality \wrt{} the Lebesgue measure.
The following theorem is a slightly more general result.
We add a convexity assumption on $\inter{\Xset}$ since it makes Poincaré constants explicit.
Recall that the Poincaré constant for a convex domain with diameter~$d>0$ is~$d/\pi$; see, \eg~\cite{bebendorf2003note}.
\begin{theorem}
	\label{thm:mse_f_fg}
	Assume that $\inter{\Xset}$ is a bounded and convex domain.
	If $0<\delta\leq\px(\x)\leq D<\infty$ for all $\x\in\inter{\Xset}$, then
	\begin{equation}
		\label{eq:comp_dist_mse}
		\expct[(f(\X)-f_g(\X))^2] \leq \Cpoinc (\lambda_{k+1}+\cdots+\lambda_n)
	\end{equation}
	with
	\begin{equation}
		\Cpoinc = \Cpoinc(\delta,D,\Xset) \defas \frac{\diam{\Xset}}{\pi} \cdot \frac{D}{\delta} > 0.
	\end{equation}
\end{theorem}
\begin{proof}
	Define
	\begin{equation}
		\inter{\Zset_\y} = \set{\z\in\R^{n-k}}{\xyz{\y,\z}\in\inter{\Xset}} \subseteq \Zset
	\end{equation}
	and note that it is convex for $\y\in\Yset^+$.
	It holds that $\diam{\inter{\Zset_\y}} \leq \diam{\Zset} \leq \diam{\Xset}$.
	Note that
	\begin{equation}
		\frac{\delta}{\py(\y)} \leq \pzy(\z|\y) \leq \frac{D}{\py(\y)}
	\end{equation}
	for $\y\in\Yset^+$ and $\z\in\inter{\Zset_\y}$.
	This justifies the following lines of computation for $\y\in\Yset^+$,
	\begin{align}
		&\expctpzy{(f(\xyz{\Y,\Z})-g(\Y))^2} \\
		&\qquad= \int_{\inter{\Zset_\y}}{(f(\xyz{\y,\z})-g(\y))^2 \, \pzy(\z|\y)\dz} \\
		&\qquad\leq \frac{D}{\py(\y)} \int_{\inter{\Zset_\y}}{(f(\xyz{\y,\z})-g(\y))^2} \dz \\
		&\qquad\leq \frac{\diam{\inter{\Zset_\y}}}{\pi} \frac{D}{\py(\y)} \int_{\inter{\Zset_\y}}{\norm{\grad^\z f(\xyz{\y,\z})}_2^2 \dz} \\
		&\qquad\leq \frac{\diam{\Xset}}{\pi} \frac{D}{\delta} \int_{\inter{\Zset_\y}}{\norm{\grad^\z f(\xyz{\y,\z})}_2^2 \, \pzy(\z|\y)\dz} \\
		&\qquad= \frac{\diam{\Xset}}{\pi} \frac{D}{\delta} \, \expctpzy{\norm{\grad^\z f(\xyz{\Y,\Z})}_2^2}. \label{eq:comp_dist_calc}
	\end{align}
	Then, combining~\eqref{eq:grad_eigvals} with~\eqref{eq:comp_dist_calc} yields the result in~\eqref{eq:comp_dist_mse}.
\end{proof}

Also, it is well-known that the Poincaré constant is one for the multivariate standard normal distribution~$\mathcal{N}(\vec{0},I)$~\cite{chen1982inequality}.
Since its density is rotationally symmetric, random variables~$\Y$ and~$\Z$ are independent and each follow again a standard normal distribution.
Hence, it holds that $C_\text{P}=1$.
For general multivariate normal distributions $\mathcal{N}(\vec{m},\Sigma)$ with mean~$\vec{m}$ and non-degenerate covariance matrix~$\Sigma$, shifting and scaling arguments give that $C_\text{P}=\lambda_{\textnormal{max}}(\Sigma)$.

\begin{remark}
	Note that the constant~$\Cpoinc$ in the previous two examples is independent of~$W$.
\end{remark}

\section{Main results}
\label{sec:main_results}
This section contains the main contribution of the manuscript which lies in an investigation of general log-concave probability measures \wrt{} their applicability for ASM.
Log-concave distributions have Lebesgue densities of the form $\px(\x) = \exp(-V(\x))$ for a convex function $V:\R^n\to(-\infty,+\infty]$.
Note that~$+\infty$ is included in the codomain of~$V$.
The conditional density $\pzy(\cdot|\y)$ for a given~$\y\in\Yset^+$ is then given by
\begin{equation}
	\pzy(\z|\y) = \frac{\exp(-V(\xyz{\y,\z}))}{\py(\y)} = \exp(-\tilde{V}_\y(\z)),
\end{equation}
where $\tilde{V}_\y(\z) \defas V(\xyz{\y,\z}) + \log(\rho_{\y}(\y))$.
Note that $\tilde{V}_\y$ inherits convexity (in $\z$) from $V$.
Bobkov~\cite{bobkov1999isoperimetric} shows that general log-concave densities satisfy a Poincaré inequality and gives lower and upper bounds on the corresponding Poincaré constant.

First, we discuss the special case of $\alpha$-uniformly convex functions~$V$ for which the corresponding density~$\px$ is known to satisfy a Poincaré inequality with universal Poincaré constant~$1/\alpha$.
However, the assumption of the density~$\px$ being of uniformly log-concave type is somewhat restrictive since it excludes distributions with heavier tails as, for example, exponential or Laplace distributions.
For this reason, we secondly investigate general log-concave densities and show that there might arise problems with this class of probability distributions due to arbitrary large Poincaré constants~$C_\Y$.
In particular, the problems and their proposed solution are exemplified on an extreme case example involving independently exponentially distributed random variables in~$n\geq2$ dimensions.

\subsection{$\alpha$-uniformly convex functions $V$}
\label{ssec:unif_conv}
\begin{definition}[$\alpha$-uniformly convex function]
	A function $V\in\mathcal{C}^2$ is said to be \textit{$\alpha$-uniformly convex}, if there is an $\alpha>0$ such that for all $\x\in\R^n$ it holds that
	\begin{equation}
		\u\tr V''(\x)\u \geq \alpha\norm{\u}_2^2
	\end{equation}
	for all $\u\in\R^n$, where $V''$ denotes the Hessian matrix of $V$.
\end{definition}

In \cite[p. 43--44]{handel2016probability}, it was shown that there is a dimension-free Poincaré constant~$1/\alpha$ for $\alpha$-uniformly log-concave~$\px$.
Note that this says nothing about the special case $\alpha=0$.
The existence of a dimension-free Poincaré constant for this special case is actually a consequence of the famous Kannan-Lovász-Simonovits conjecture; see, \eg~\cite{alonso2015approaching,lee2018kannan}.
However, since we need a Poincaré inequality for $\pzy(\cdot|\y)$, $\y\in\Yset^+$, we have to prove the following lemma similar to~\cite[Subsection~7.2]{zahm2018certified}
\begin{lemma}
	If $\px$ is $\alpha$-uniformly log-concave, then $\pzy(\cdot|\y)$ is $\alpha$-uniformly log-concave for each~$\y\in\Yset^+$.
\end{lemma}
\begin{proof}
	Let $\y\in\Yset^+$.
	Recall that $\pzy(\z|\y) = \exp(-\tilde{V}_\y(\z))$ for a convex function~$\tilde{V}_\y(\z) \defas V(\xyz{\y,\z}) + \log(\rho_{\y}(\y))$.
	The Hessian matrix~$\tilde{V}_\y''(\z)$ (\wrt~$\z$) computes to
	\begin{equation}
		\tilde{V}_\y''(\z) = W_2\tr\,V''(\xyz{\y,\z})\,W_2.
	\end{equation}
	Choose $\w\in\R^{n-k}$ arbitrarily.
	Then, for every $\z\in\R^{n-k}$, it holds that
	\begin{align}
		\w\tr \tilde{V}_\y''(\z)\w &= (W_2\w)\tr V''(\xyz{\y,\z})\,(W_2\w) \\
		&\geq \alpha\norm{W_2\w}_2^2 = \alpha\norm{\w}_2^2.
	\end{align}
\end{proof}
Since $\pzy(\cdot|\y)$ inherits the universal Poincaré constant~$1/\alpha$ from~$\px$, the result in~\eqref{eq:act_subsp_ineq_poinc} also holds for $\alpha$-uniformly log-concave densities with~$\Cpoinc=1/\alpha$ (independent of~$W$) which is similar to~\cite[Corollary~2]{zahm2018certified}.

For example, $\alpha$-uniformly log-concave densities comprise multivariate normal distributions~$\mathcal{N}(\vec{m},\Sigma)$ with mean~$\mathbf{m}$ and covariance matrix~$\Sigma$ ($\alpha=1/\lambda_{\textnormal{max}}(\Sigma)$).
However, distributions that satisfy the assumption only for $\alpha=0$ as, \eg Weibull distributions with the exponential distribution as a special case or Gamma distributions with shape parameter~$\beta\geq1$, only belong to the class of general log-concave distributions.

\subsection{General convex functions $V$}
\label{ssec:gen_convex_fcts}
Since we cannot make use of a universal dimension-free Poincaré constant involving general convex functions~$V:\R^n\to(-\infty,+\infty]$, we look at them more closely in this subsection.
Recall that $\pzy(\z|\y)=\exp(-\tilde{V}_\y(\z))$, $\y\in\Yset^+$, for a convex function $\tilde{V}_\y$.
We have to deal with the fact that the essential supremum of the random Poincaré constant~$C_\Y$ of $\pzy(\cdot|\Y)$ does possibly not exist.
A corresponding example is given in Subsection~\ref{sssec:main_results_exp2d}.
In the step from \eqref{eq:poinc_pzy} to \eqref{eq:supCy_exists}, we have applied Hölder's inequality with Hölder conjugates $(p,q)=(+\infty,1)$.
Since this is not possible for unbounded random variables~$C_\Y$, we can only show a weaker result.
\begin{lemma}
	\label{lem:CpoincepsW}
	If $\norm{\grad f}_2^2 \leq L$ for some constant~$L>0$, then
	\begin{equation}
		\expct[(f(\X)-f_g(\X))^2] \leq \CpoincepsW (\lambda_{k+1}+\cdots+\lambda_n)^{1/(1+\eps)},
	\end{equation}
	where
	\begin{equation}
		\label{eq:CpoincepsW}
		\CpoincepsW = \CpoincepsW (\eps,n,k,L,W,\px) \defas L^{\eps/(1+\eps)}\expct[C_\Y^{(1+\eps)/\eps}]^{\eps/(1+\eps)}.
	\end{equation}
\end{lemma}
\begin{proof}
	The boundedness of~$\grad f$ implies that also~$\norm{\grad^\z f}_2^2\leq L$.
	Choosing a weaker pair of conjugates~$(p,q)=((1+\eps)/\eps,1+\eps)$, $\eps>0$, we compute
	\begin{align}
		&\expct[C_\Y\,\expctpzy{\norm{\grad^\z f(\xyz{\Y,\Z})}_2^2}] \\
		&\qquad\leq \expct[C_\Y^p]^{1/p} \, \expct[\expctpzy{\norm{\grad^\z f(\xyz{\Y,\Z})}_2^2}^q]^{1/q} \label{eq:hoelder_weak} \\
		&\qquad= \expct[C_\Y^{(1+\eps)/\eps}]^{\eps/(1+\eps)} \, \expct[\expctpzy{\norm{\grad^\z f(\xyz{\Y,\Z})}_2^2}^{1+\eps}]^{1/(1+\eps)} \label{eq:hoelder_eps} \\
		&\qquad\leq L^{\eps/(1+\eps)} \, \expct[C_\Y^{(1+\eps)/\eps}]^{\eps/(1+\eps)} \, \expct[\expctpzy{\norm{\grad^\z f(\xyz{\Y,\Z})}_2^2}]^{1/(1+\eps)} \\
		&\qquad\leq L^{\eps/(1+\eps)} \, \expct[C_\Y^{(1+\eps)/\eps}]^{\eps/(1+\eps)} \, (\lambda_{k+1}+\cdots+\lambda_n)^{1/(1+\eps)} \label{eq:CpoincepsW_grad_eigvals} \\
		&\qquad= \CpoincepsW \, (\lambda_{k+1}+\cdots+\lambda_n)^{1/(1+\eps)}.
	\end{align}
	The step in~\eqref{eq:CpoincepsW_grad_eigvals} uses~\eqref{eq:grad_eigvals}.
	The result follows by~\eqref{eq:mse_f_fg_fX_fYZ} and~\eqref{eq:poinc_pzy}.
\end{proof}
\begin{remark}
	The previous lemma requires the gradient of~$f$ to be uniformly bounded, an assumption that is not needed in~\cite{constantine2014active} and~\cite{zahm2018certified}.
	
	However, first, applying ASM, in the sense that the matrix~$C$ from~\eqref{eq:C} is estimated by a finite Monte Carlo sum, requires the same assumption to prove results on corresponding approximations of eigenvalues~$\lambda_i$ and eigenvectors~$\w_i$; see~\cite{constantine2014computing} and~\cite[Section~3.3]{constantine2015active}.
	
	Secondly, this assumption can be weakened by applying another Hölder's inequality analogous to~\eqref{eq:hoelder_weak}.
	Indeed, for~$\eps\in(0,1)$, we would get
	\begin{align}
		&\expct[\expctpzy{\norm{\grad^\z f(\xyz{\Y,\Z})}_2^2}^{1+\eps}]^{1/(1+\eps)} \\
		&\qquad \leq \expct[\expctpzy{\norm{\grad^\z f(\xyz{\Y,\Z})}_2^2}^{1/(1-\eps)}]^{(1-\eps)/(1+\eps)} \\
		&\qquad\qquad \cdot \expct[\expctpzy{\norm{\grad^\z f(\xyz{\Y,\Z})}_2^2}]^{\eps/(1+\eps)}.
	\end{align}
	Since
	\begin{align}
		&\expct[\expctpzy{\norm{\grad^\z f(\xyz{\Y,\Z})}_2^2}^{1/(1-\eps)}] \\
		&\qquad \leq \expct[\expctpzy{\norm{\grad^\z f(\xyz{\Y,\Z})}_2^{2/(1-\eps)}}] \\
		&\qquad \leq \expct[\norm{\grad^\x f(\X)}_2^{2/(1-\eps)}],
	\end{align}
	we would only require~$\norm{\grad^\x f(\X)}_2^{2/(1-\eps)}$ to be integrable.
	What we, however, would have to accept in this case, is the resulting weaker order~$\eps/(1+\eps)$ in the inactive trace.
\end{remark}
The $L$- and $\px$-dependence of~$\CpoincepsW$ is notationally neglected in the following.
If possible, we can choose a suitable~$\eps>0$ to get $\expct[C_\Y^{(1+\eps)/\eps}]<\infty$ and thus a finite constant~$\CpoincepsW$.
Note that we lose first order in the eigenvalues from the inactive subspace, but have instead order~$1/(1+\eps)<1$.
Of course, the constant~$\CpoincepsW$ could get arbitrarily large as~$\eps\to0$, but this strongly depends on~$W$ and the moments of~$C_\Y$; see the example given in Subsection~\ref{sssec:main_results_exp2d}.

It is known by Bobkov~\cite[Eqs.~(1.3), (1.8) and p.~1906]{bobkov1999isoperimetric} that there exists a (dimensionally dependent) Poincaré constant~$C_\y$ for a general log-concave density~$\pzy(\cdot|\y)$ that is bounded from below and above by
\begin{equation}
	\label{eq:C_y_bobkov}
	\begin{split}
		&\expct[(\norm{\Z-\z_0}_2-\expct[\norm{\Z-\z_0}_2\|\Y=\y])^2\|\Y=\y] \leq C_\y \\
		&\qquad\leq K\, \expct[\norm{\Z-\z_0}_2^2\|\Y=\y] \\
		&\qquad= K \sum_{i=1}^{n-k}{\varcond{Z_i}{\Y=\y}},
	\end{split}
\end{equation}
where $\z_0\defas\expctcond{\Z}{\Y=\y}$ and $K=432$ \cite[Eqs.~(1.8) and~(3.4)]{bobkov1999isoperimetric} is a universal constant.
To the authors' knowledge, the constant~$C_\y$ is the best available.
We provide a scenario in Subsection~\ref{sssec:main_results_exp2d} (\grqq{}Rotation by~$\theta=\pi/4$") in which the lower bound viewed as a random variable has no finite essential supremum implying the same for~$C_\Y$.

However, to make use of Lemma~\ref{lem:CpoincepsW}, we need to investigate the involved constant~$\CpoincepsW(\eps,n,k)$.
\begin{lemma}
	It holds that
	\begin{equation}
		\label{eq:bound_Cy_eps}
		\expct[C_\Y^{(1+\eps)/\eps}]^{\eps/(1+\eps)} \leq K (n-k)^{1/(1+\eps)} \CvarW,
	\end{equation}
	where
	\begin{equation}
		\label{eq:CvarW}
		\CvarW = \CvarW(\eps,n,k,W) \defas \left(\sum_{i=1}^{n-k}\expct[\var(Z_i|\Y)^{(1+\eps)/\eps}] \right)^{\eps/(1+\eps)}.
	\end{equation}
\end{lemma}
\begin{proof}
Using Jensen's inequality for weighted sums, it follows that
\begin{align}
	\expct[C_\Y^{(1+\eps)/\eps}] &\leq K^{(1+\eps)/\eps} (n-k)^{(1+\eps)/\eps}\frac{1}{n-k} \sum_{i=1}^{n-k}\expct[\var(Z_i|\Y)^{(1+\eps)/\eps}] \\
	&= K^{(1+\eps)/\eps}(n-k)^{1/\eps}\sum_{i=1}^{n-k}\expct[\var(Z_i|\Y)^{(1+\eps)/\eps}].
\end{align}
The result follows.
\end{proof}

Eventually, we get
\begin{equation}
	\label{eq:CpoincepsW_CvarW_depend}
	\CpoincepsW(\eps,n,k) \leq L^{\eps/(1+\eps)} K (n-k)^{1/(1+\eps)} \CvarW(\eps,n,k).
\end{equation}
As before, we can remove the dependence of~$\CvarW$ on~$W=W(f)$ by considering the supremum over all orthogonal matrices.
That is, we define
\begin{equation}
	\label{eq:Cpoinceps}
	\Cpoinceps \defas \sup_{W \text{orth.}} \CpoincepsW
\end{equation}
and
\begin{equation}
	\label{eq:Cvar}
	\Cvar \defas \sup_{W \text{orth.}} \CvarW,
\end{equation}
and get
\begin{equation}
	\expct[(f(\X)-f_g(\X))^2] \leq \Cpoinceps (\lambda_{k+1}+\cdots+\lambda_n)^{1/(1+\eps)},
\end{equation}
provided the constant~$\Cpoinceps=\Cpoinceps(\eps,n,k,L,\px)$ exists.

For~$\Cvar$, we argue that it is actually enough to take the supremum only over the set of rotation matrices.
Indeed, any orthogonal matrix~$W$ is either a \textit{proper} ($\det W=1$) or an \textit{improper} ($\det W=-1$) rotation which is the combination of a proper rotation and an inversion of the axes; see, \eg~\cite{kinsey2006symmetry,morawiec2003orientations}.
However, since the constant~$\CvarW$ from~\eqref{eq:CvarW} is invariant to inversions of the axes, it holds that
\begin{equation}
	\label{eq:supW_supR}
	\sup_{W\,\text{orth.}} \CvarW = \sup_{R\,\text{rot.}} C_{\var,R}.
\end{equation}
This equality is exploited in the next subsection.

\subsection{Independently exponentially distributed random variables as an extreme case}
\label{ssec:exp_dist_case}
In this subsection, we take a closer look at independently exponentially distributed random variables in~$n\geq2$ dimensions as an example for a general log-concave distribution.
In particular, we use the lower bound of Bobkov from~\eqref{eq:C_y_bobkov} in Subsection~\ref{sssec:main_results_exp2d} to show that there exists a scenario in which the random Poincaré constant~$C_\Y$ does not have an essential supremum implying that~$\Cpoinc$ from~\eqref{eq:C_poinc} does not exist.
Therefore, the quantity~$C_{\var}$ from~\eqref{eq:Cvar} is investigated in Subsections~\ref{sssec:main_results_exp2d} and~\ref{sssec:main_results_expnd} to derive a (finite) upper bound for~$\Cpoinceps$ from~\eqref{eq:Cpoinceps} in this special case.

We regard a random vector~$\X=(X_1,\ldots,X_n)\tr$ whose components are independently exponentially distributed with unit rates $\nu_i=1$, $i=1,\ldots,n$ and will see that investigations with unit rates are sufficient to derive statements also involving other rates.
The distribution of~$\X$ has the density
\begin{equation}
	\label{eq:dens_exp_n}
	\px(\x) = \begin{cases}\exp(-x_1-\cdots-x_n) &\text{if $\x=(x_1,\ldots,x_n)\tr\in\R_{\geq0}^n$}, \\0 &\text{otherwise}.\end{cases}
\end{equation}
That is, in this case $\Xset=\R_{\geq0}^n$ and
\begin{equation}
	V(\x) = \begin{cases}x_1+\cdots+x_n &\text{if $\x=(x_1,\ldots,x_n)\tr\in\R_{\geq0}^n$}, \\+\infty &\text{otherwise}.\end{cases}
\end{equation}
Note that $V$ is convex.

Since we are interested in~$\Cvar$ as a supremum over all orthogonal matrices, we assume that, in this subsection,~$W=\begin{pmatrix}W_1&W_2\end{pmatrix}$ is an \textit{arbitrary} orthogonal matrix not depending on~$f$ and~$\px$.
Indeed, as the equality in~\eqref{eq:supW_supR} motivates, we can further assume that~$W$ is a rotation matrix.

\subsubsection{$2$ dimensions}
\label{sssec:main_results_exp2d}
The joint density of two independently exponentially distributed random variables $X_1$ and $X_2$ both with unit rate is
\begin{equation}
	\px(x_1,x_2) = \begin{cases}
		\exp(-x_1-x_2) \quad&\text{if $x_1,x_2\geq0$}, \\
		0 &\text{otherwise}.
	\end{cases}
\end{equation}
First, let us regard a rotation of the two-dimensional Cartesian coordinate system by a general angle~$\theta\in[-\pi,\pi)$ to a coordinate system for $(y,z)$, \ie
\begin{equation}
	\begin{pmatrix}x_1\\x_2\end{pmatrix} = R_\theta \begin{pmatrix}y\\z\end{pmatrix}
\end{equation}
for a rotation matrix
\begin{equation}
	W = R_\theta \defas \begin{pmatrix}\cos\theta & -\sin\theta \\ \sin\theta & \cos\theta\end{pmatrix}.
\end{equation}
That is, in two dimensions, it holds that
\begin{equation}
	\Cvar = \sup_{\theta\in[-\pi,\pi)} C_{\var,R_\theta}.
\end{equation}
Subsequently, we look at the special case $\theta=\pi/4$ as an example for an unbounded Poincaré constant~$C_y$ of $\rho_{z|y}(\cdot|y)$.
Variables are written in thin letters in this subsection since they denote real values and not multidimensional vectors.

Note that the bound from~\eqref{eq:bound_Cy_eps} in this two-dimensional setting becomes
\begin{equation}
	\expct[C_Y^{(1+\eps)/\eps}]^{\eps/(1+\eps)} \leq K \CvarW(\eps,2,1)
\end{equation}
with
\begin{equation}
	\CvarW(\eps,2,1) = \expct[\var(Z|Y)^{(1+\eps)/\eps}]^{\eps/(1+\eps)}.
\end{equation}

\subsubsection*{Rotation by general $\theta$}
Let~$\theta\in[-\pi,\pi)$.
Then, the joint density of~$(Y,Z)$ is
\begin{align}
	\rho_{Y,Z}(y,z) &= \exp(-(y\cos\theta-z\sin\theta) -(y\sin\theta+z\cos\theta)) \\
	&= \exp(-(\cos\theta+\sin\theta)y -(\cos\theta-\sin\theta)z).
\end{align}
for~$(y,z)$ with $\xyz{y,z}\in\R_{\geq0}^2$ and zero otherwise.
If we define $a_\theta^+\defas\cos\theta+\sin\theta$ and $a_\theta^-\defas\cos\theta-\sin\theta$, we have
\begin{equation}
	\label{eq:pyz_theta}
	\rho_{Y,Z}(y,z) = \begin{cases}
			\exp(-a_\theta^+y-a_\theta^-z) \quad&\text{if $\xyz{y,z}\in\R_{\geq0}^2$}, \\
			0 & \text{otherwise}.
		\end{cases}
\end{equation}
Fig.~\ref{fig:exp2d} illustrates the situation for a positive (Fig.~\ref{fig:exp2d_a}) and a negative (Fig.~\ref{fig:exp2d_b}) angle~$\theta$.
\begin{figure*}
	\centering
	\begin{subfigure}{0.45\linewidth}
		\centering
		\includegraphics[width=\linewidth]{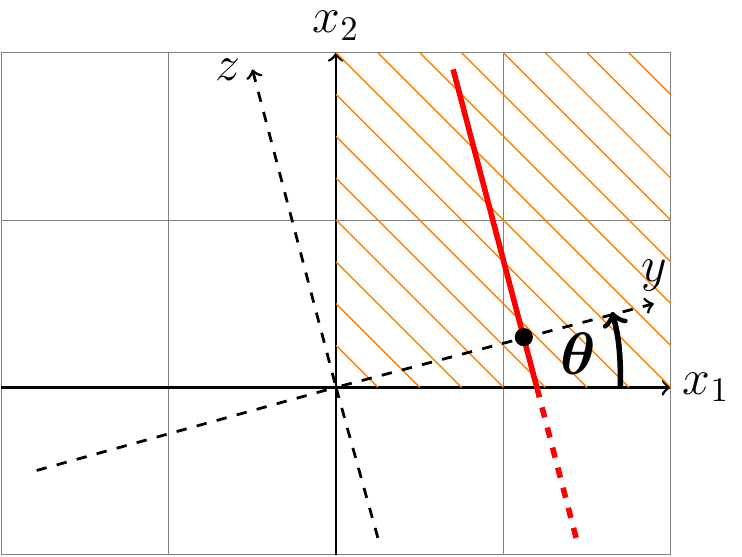}
		\caption{}
		\label{fig:exp2d_a}
	\end{subfigure}
	\hspace{0.5cm}
	\begin{subfigure}{0.45\linewidth}
		\centering
		\includegraphics[width=\linewidth]{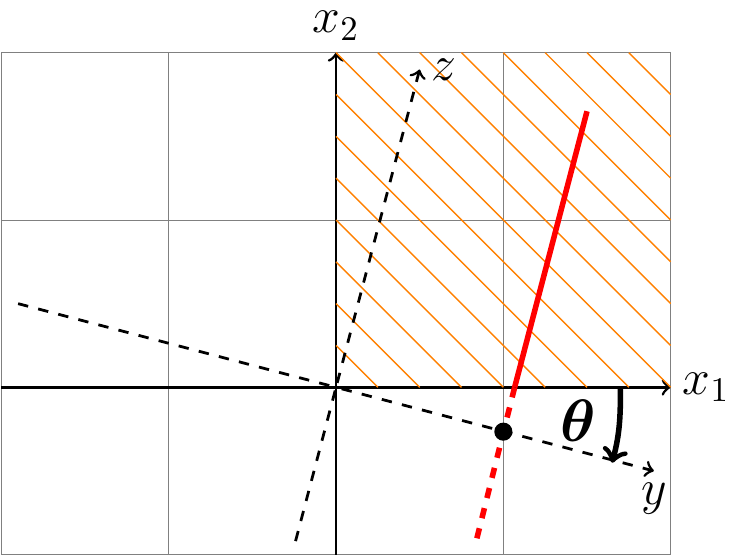}
		\caption{}
		\label{fig:exp2d_b}
	\end{subfigure}
	\caption{Rotations of the coordinate system with a positive (a) and a negative (b) angle.
		The orange lines depict contour lines in the support of~$\px$.
		The red lines show the values of $(y,z)$ for a given $y$.
		Their solid parts mark regions within the support of~$\px$, whereas the dashed parts identify values with density zero.}
	\label{fig:exp2d}
\end{figure*}

The interval of investigation for $\theta\in[-\pi,\pi)$ can be reduced by reasons of periodicity and symmetry.
First, note that the map 
\begin{equation}
	\label{eq:def_Qeps}
	Q_\eps(\theta)\defas C_{\var,R_\theta}(\eps,2,1),
\end{equation}
is $\pi$-periodic in $\theta$ since an additional rotation by~$\pi$ corresponds to changing signs of~$y$ and~$z$ which is not important for integrals in~$Q_\eps$.
Hence, it suffices to consider $\theta\in[-\pi/2,\pi/2)$.
Secondly, from Fig.~\ref{fig:exp2d} it can be deduced that~$Q_\eps$, as a map of~$\theta$, is symmetric around~$-\pi/4$ in $[-\pi/2,0]$ and symmetric around~$\pi/4$ in $[0,\pi/2)$.
This fact is also shown in Fig.~\ref{fig:theta_curves}. That is, it is enough to investigate angles~$\theta\in[-\pi/4,\pi/4]$.

\begin{figure}
	\centering
	\includegraphics[width=0.75\linewidth]{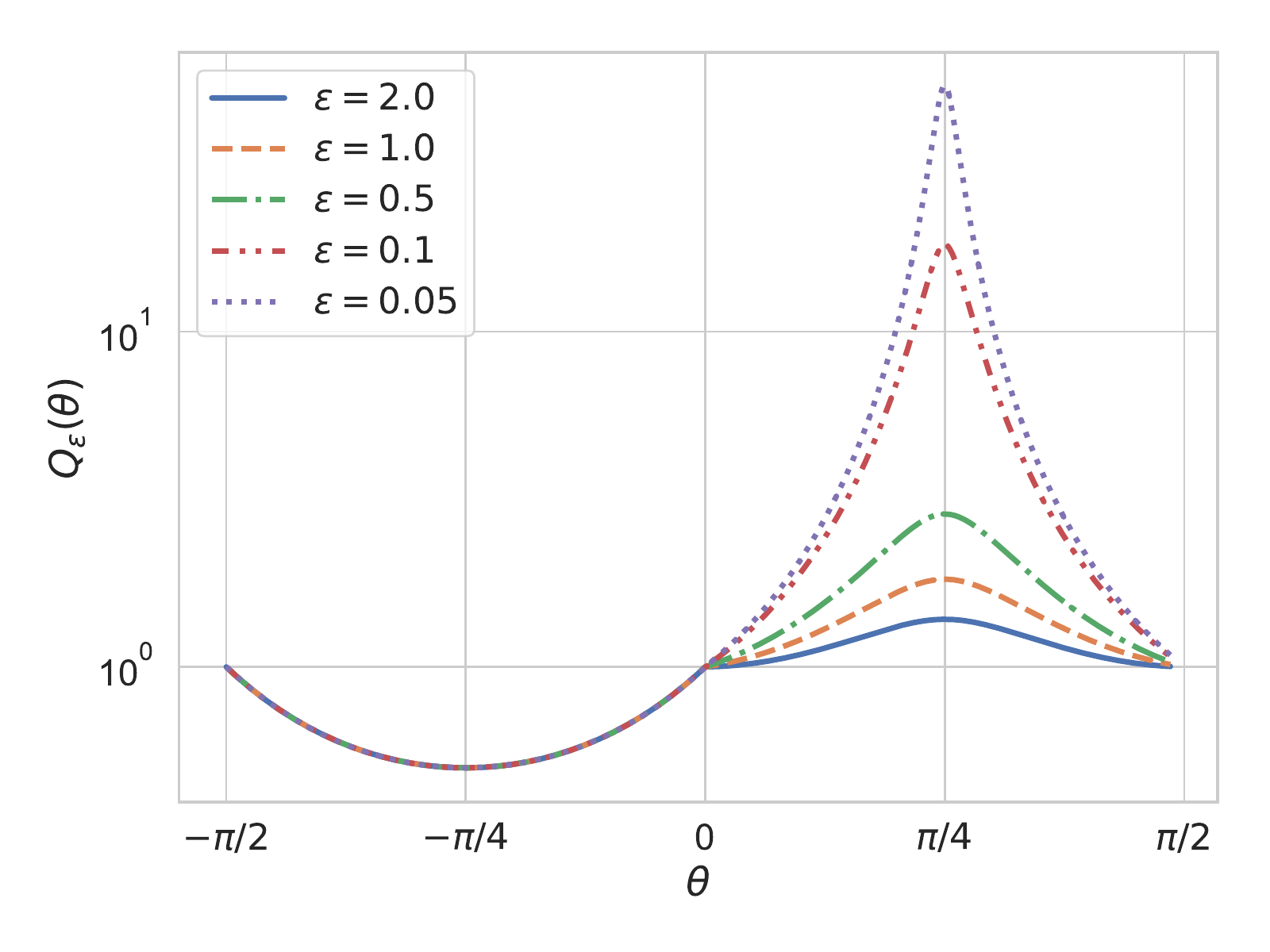}
	\caption{Illustration of symmetries in $\theta$ of the map~$Q_\eps(\theta)$ for several $\eps>0$.}
	\label{fig:theta_curves}
\end{figure}

For the computation of integrals in $Q_\eps(\theta)$, $\theta\in[-\pi/4,\pi/4]$, it is necessary, for a given $y$, to determine boundaries~$\ell_0(y)$ and $\ell_1(y)$ of intervals for $z$ that lie in the support of the joint density~$\rho_{Y,Z}(y,z)$ (see the thick solid lines in Fig.~\ref{fig:exp2d}).
The integrals in~$Q_\eps(\theta)$ are computed using the computer algebra system~\wm{}~\cite{mathematica}.
The computation requires to treat the cases~$\theta\in[-\pi/4,0)$ and~$\theta\in[0,\pi/4]$ differently (see Fig.~\ref{fig:exp2d}).

For negative $\theta\in[-\pi/4,0)$ and arbitrary $y\in\R$, we have that
\begin{equation}
	\ell_0(y) = \left.\begin{cases}
	\abs{y}\cot(\abs{\theta}) &\text{if $y<0$} \\
	y\tan(\abs{\theta}) &\text{if $y\geq0$}
	\end{cases}
	\right\rbrace = \abs{y}\tan(\abs{\theta})^{\text{sgn}(y)}
\end{equation}
and $\ell_1(y)=\infty$, \ie
\begin{equation}
	\rho_{Y,Z}(y,z) = \exp(-a_\theta^+y - a_\theta^-z) \cdot \ind_{[\ell_0(y),\ell_1(y)]}(z).
\end{equation}
We compute that
\begin{equation}
	\varcond{Z}{Y=y} = (\cos(\abs{\theta})+\sin(\abs{\theta}))^{-2}
\end{equation}
which is constant in $y$ and yields
\begin{equation}
	Q_\eps(\theta) = C_{\var,R_\theta}(\eps,2,1) = (\cos(\abs{\theta})+\sin(\abs{\theta}))^{-2}.
\end{equation}
Note that this explains the left part of the graph of $Q_\eps(\theta)$ in Fig.~\ref{fig:theta_curves} which shows that~$Q_\eps(\theta)$ does not depend on~$\eps$ for~$\theta\in[-\pi/2,0)$.

For non-negative $\theta\in[0,\pi/4]$ and a given $y\geq0$, the boundaries are computed to $\ell_0(y)=-y\tan(\theta)$ and $\ell_1(y)=y\cot(\theta)$, \ie
\begin{equation}
	\rho_{Y,Z}(y,z) = \exp(-a_\theta^+y - a_\theta^-z) \cdot \ind_{[0,\infty)}(y) \cdot \ind_{[\ell_0(y),\ell_1(y)]}(z).
\end{equation}
We compute that
\begin{align}
	&\varcond{Z}{Y=y} \nonumber \\
	&\quad= \frac{a_\theta}{8b_\theta^2} \left( \frac{1-2\exp(b_\theta y)+\exp(2b_\theta y)-8\exp(b_\theta y)y^2(1-d_\theta)}{(\exp(b_\theta y)-1)^2} - c_\theta \right)
\end{align}
for $a_\theta\defas\csc(\theta)^4\sec(\theta)^4$, $b_\theta\defas\sec(\theta)-\csc(\theta)$, $c_\theta\defas\cos(4\theta)$, and $d_\theta\defas\sin(2\theta)$.
$\varcond{Z}{Y=y}$ can actually be bounded in~$y$ for $\theta\in[0,\pi/4)$.
Indeed, since $d_\theta\in[0,1)$, it holds that $1-d_\theta\in(0,1]$ implying that $8\exp(b_\theta y)y^2(1-d_\theta)>0$.
It follows that
\begin{align}
	\var(Z|Y) &\leq \frac{a_\theta}{8b_\theta^2} \left( \frac{1-2\exp(b_\theta y)+\exp(2b_\theta y)}{(\exp(b_\theta y)-1)^2}-c_\theta \right) \\
	&= \frac{a_\theta}{8b_\theta^2} \left( \frac{(\exp(b_\theta y)-1)^2}{(\exp(b_\theta y)-1)^2}-c_\theta \right) \\
	&= \frac{a_\theta(1-c_\theta)}{8b_\theta^2}.
\end{align}
Fig.~\ref{fig:y_varzy} illustrates the boundedness of $\varcond{Z}{Y=y}$ and additionally shows that it approaches the unbounded function~$y\mapsto y^2/3$ as~$\theta\to\pi/4$.
Hence, for~$\theta\in[0,\pi/4)$, it holds that
\begin{equation}
	Q_\eps(\theta) = C_{\var,R_\theta}(\eps,2,1) \leq \frac{a_\theta(1-c_\theta)}{8b_\theta^2}.
\end{equation}
This bound is itself unbounded in~$\theta$ since~$b_\theta\to0$ and $a_\theta(1-c_\theta)\to32$ as~$\theta\to\pi/4$ implying that we can see~$\theta=\pi/4$ as a special case.
This assessment is also supported by Fig.~\ref{fig:eps_curves}.
\begin{figure*}
	\centering
	\begin{subfigure}{0.495\linewidth}
		\centering
		\includegraphics[width=\linewidth]{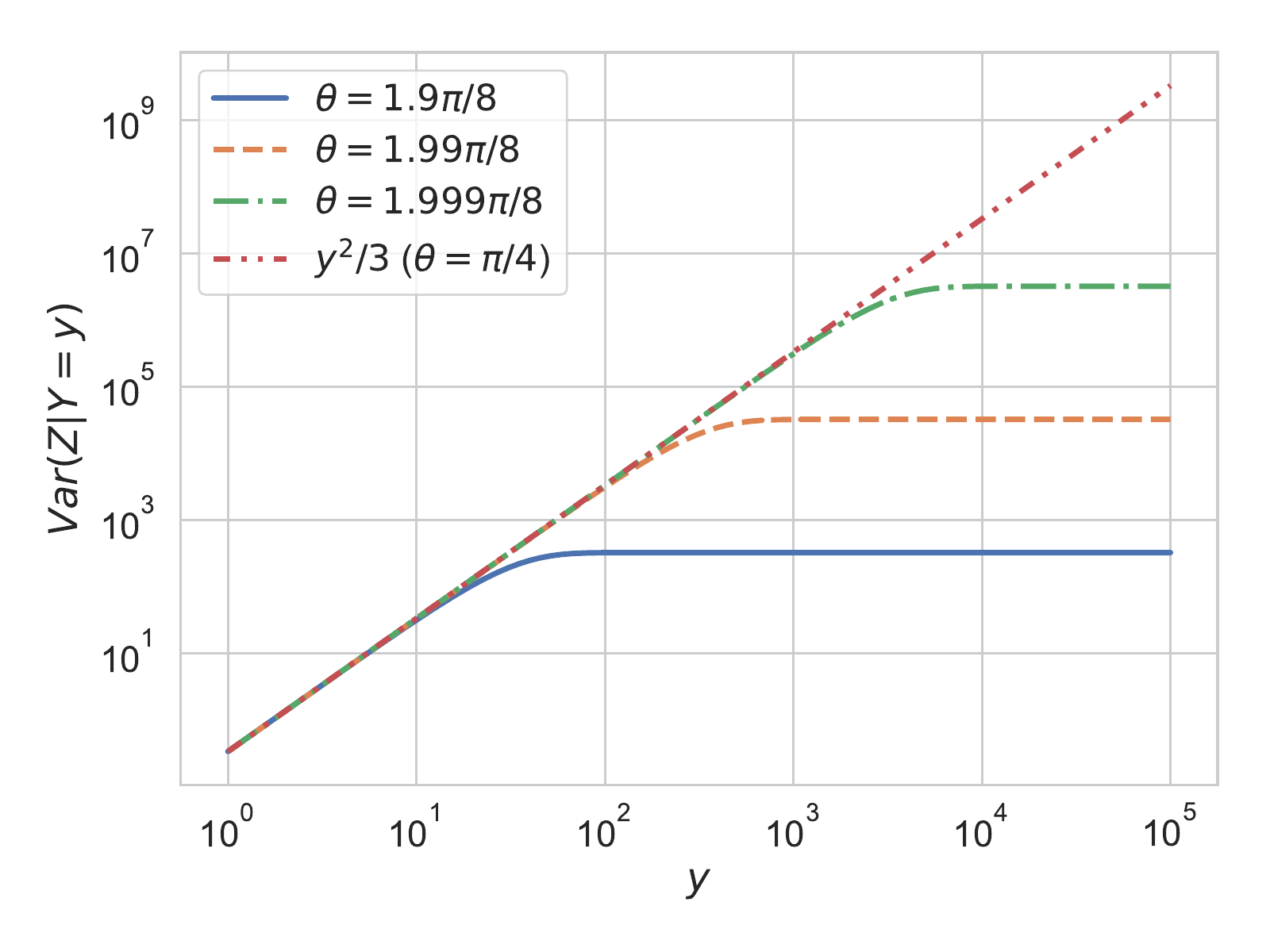}
		\caption{}
		\label{fig:y_varzy}
	\end{subfigure}
	\begin{subfigure}{0.495\linewidth}
		\centering
		\includegraphics[width=\textwidth]{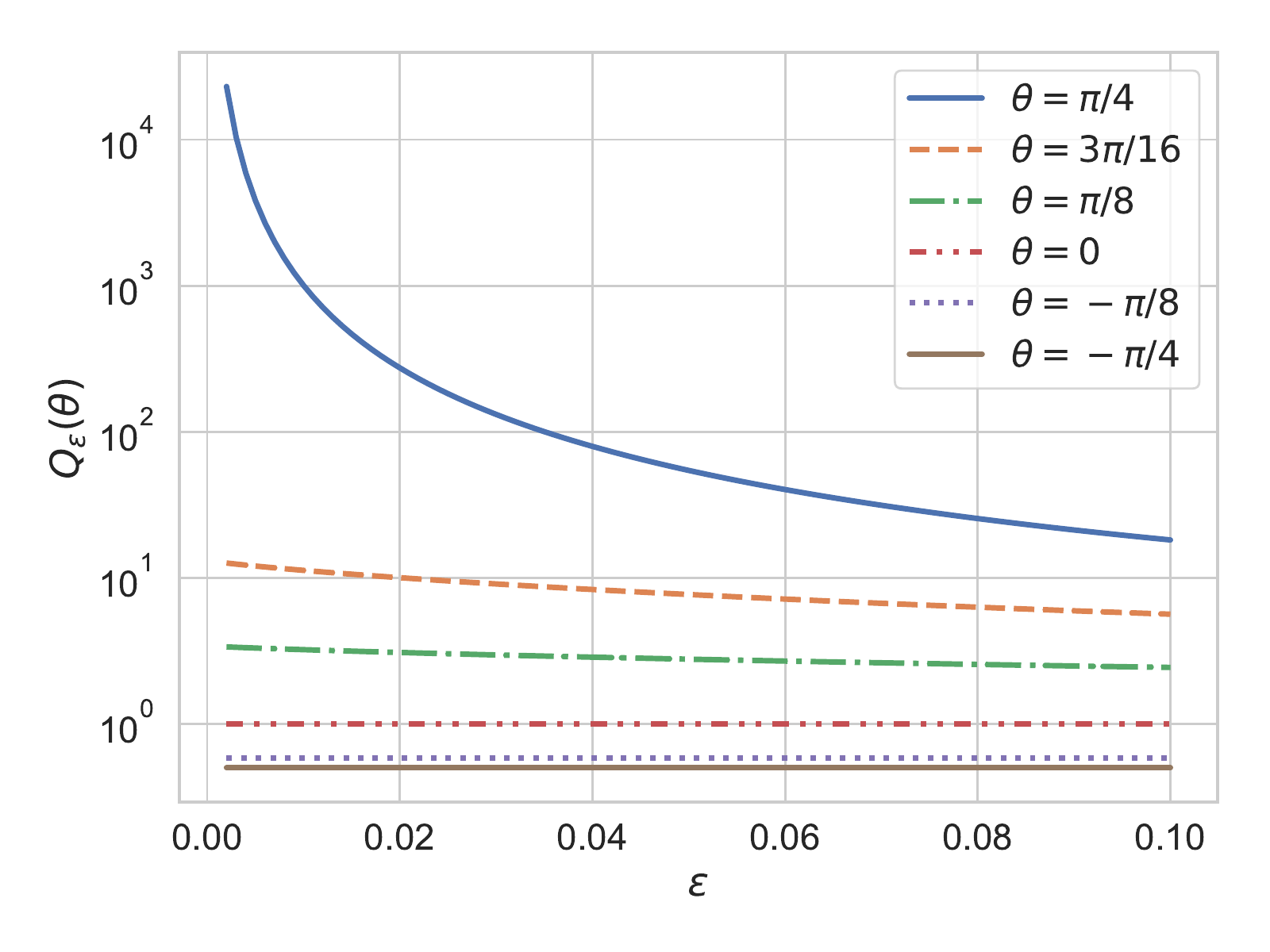}
		\caption{}
		\label{fig:eps_curves}
	\end{subfigure}
	\caption{(a) The log-log plot of the map $y\mapsto\varcond{Z}{Y=y}$ shows that it is bounded for angles~$\theta\in[0,\pi/4)$, but approaching the unbounded function~$y^2/3$, which corresponds to~$\theta=\pi/4$, as $\theta\to\pi/4$. \\
	(b) The plot shows the map~$\eps\mapsto Q_\eps(\theta)$ for several angles~$\theta$.
	Also, it illustrates the fact that $\theta=\pi/4$ is a special case for which $Q_\eps(\theta)$ can get arbitrarily large.}
\end{figure*}
In particular, note that
\begin{equation}
	\Cvar = C_{\var,R_{\pi/4}}.
\end{equation}


\subsubsection*{Rotation by $\theta=\pi/4$}
A rotation of $45^\circ$, \ie $\theta=\pi/4$ and~$W=R_{\pi/4}$, is a limit case since $a_{\pi/4}^-$ from~\eqref{eq:pyz_theta} becomes zero.
The joint density for $Y$ and $Z$ is then
\begin{equation}
	\rho_{Y,Z}(y,z) = \exp(-\sqrt{2}y) \cdot \ind_{[0,\infty)}(y) \cdot \ind_{[-y,y]}(z).
\end{equation}
A graphical illustration of this case is given in Fig.~\ref{fig:exp2d_pi4}.
Consequently, the marginal distribution of $Y$ is
\begin{equation}
	\rho_Y(y) = \int_{-\infty}^{\infty}\rho_{Y,Z}(y,z)\d{z} = 2y\exp(-\sqrt{2}y) \cdot \ind_{[0,\infty)}(y)
\end{equation}
and the conditional density $\rho_{Z|Y}(\cdot|y)$ computes to
\begin{equation}
	\rho_{Z|Y}(z|y)=\frac{\ind_{[-y,y]}(z)}{2y}
\end{equation}
for $y>0$.
Note that $\rho_{Z|Y}(\cdot|y)$ is the density of a uniform distribution on the interval~$[-y,y]$.
\begin{figure}
	\centering
	\includegraphics{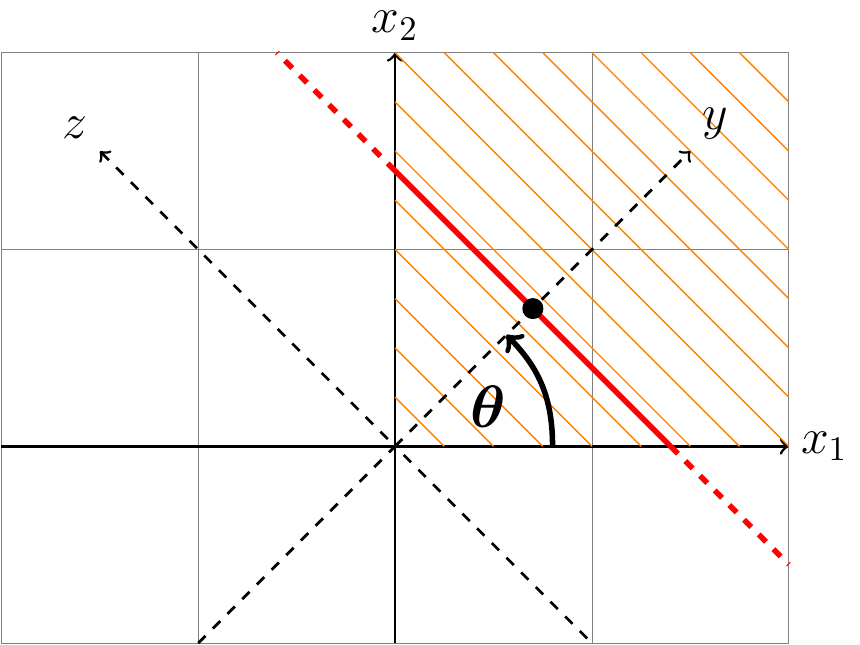}
	\caption{Exponential distribution in $2$ dimensions with a coordinate system rotated by $45^\circ$.
		The orange lines depict the contour levels of the distribution in the support of~$\px$.
		The solid red line marks the interval of the uniform distribution of $Z\|Y=y$ for $y>0$.}
	\label{fig:exp2d_pi4}
\end{figure}
For~$Y>0$, it follows that
\begin{equation}
	\label{eq:varzy_pi4}
	\var(Z|Y) = (2Y)^{-1}\int_{-Y}^{Y}{z^2\d{z}} = Y^2/3,
\end{equation}
which is the expression that variances of $Z|Y$ for other angles~$\theta^*$ approach to as $\theta^*\to\pi/4$ (see Fig.~\ref{fig:y_varzy}).

Note that the lower bound from~\eqref{eq:C_y_bobkov} for $C_Y$ in this case becomes
\begin{equation}
	\expct[(\abs{Z}-\expct[\abs{Z}\,|\,Y])^2\,|\,Y] = \var(\abs{Z}\,|\,Y) = Y^2/12,
\end{equation}
since $\abs{Z}\,|\,Y\sim\mathcal{U}([0,Y])$ and, hence, its distribution is not compactly supported implying the same for the distribution of~$C_Y$.
Therefore, we found a scenario in which the constants~$\CpoincW$ and~$\Cpoinc$ indeed do not exist.

However, there is still a chance that the constants~$\CpoincepsW$ and~$\Cpoinceps$ from~\eqref{eq:CpoincepsW} and, respectively,~\eqref{eq:Cpoinceps} exist.
It holds that
\begin{align}
	\Cvar(\eps,2,1) = C_{\var,R_{\pi/4}}(\eps,2,1) &= \frac{1}{3} \expct[Y^{2+2/\eps}]^{\eps/(1+\eps)}
\end{align}
implying that the constant~$\Cpoinceps(\eps,2,1)$ can be bounded from above by
\begin{equation}
	\Cpoinceps(\eps,2,1) \leq L^{\eps/(1+\eps)} \frac{K}{3} \expct[Y^{2+2/\eps}]^{\eps/(1+\eps)}.
\end{equation}
For example, choosing~$\eps=2$ would give
\begin{equation}
	\Cpoinceps(2,2,1) \leq 2K\left(\frac{L^2}{3}\right)^{1/3}.
\end{equation}

\subsubsection{$n$ dimensions}
\label{sssec:main_results_expnd}
This subsection aims to generalize the results of the previous subsection, \ie we investigate the constant~$\Cpoinceps$ from~\eqref{eq:Cpoinceps} for $n$ independently exponentially distributed random variables.

Motivated by the two-dimensional case, we regard the rotation of the coordinate system by a matrix~$W=R^*$ that rotates the vector $(1,0,\ldots,0)\tr\in\R^n$ to $(1/\sqrt{n},\ldots,1/\sqrt{n})\tr\in\R^n$.
Note that in the two-dimensional case, a rotation by $\theta=\pi/4$ corresponds to a matrix rotating $(1,0)\tr$ to $(1/\sqrt{2},1/\sqrt{2})\tr$.
This is the worst case in the sense that $Z_i|\Y$ is uniformly distributed for each component $Z_i$ in $\Z=(Z_1,\ldots,Z_{n-k})\tr$ and hence, similar to the two-dimensional case, the conditional variance of~$Z_i|\Y$ has no finite essential supremum.
In the context from above, it holds that
\begin{equation}
	\Cvar(\eps,n,k) = C_{\var,R^*}(\eps,n,k).
\end{equation}
The following theorem studies this case and investigates the dimensional dependence of the involved constant.

\begin{theorem}
	For~$\px$ as in~\eqref{eq:dens_exp_n}, it holds that
	\begin{equation}
		\expct[(f(\X)-f_g(\X))^2] \leq C_{\exp^n} \, (\lambda_{k+1}+\cdots+\lambda_n)^{1/(1+\eps)}
	\end{equation}
	for a constant
	\begin{equation}
	C_{\exp^n}=C_{\exp^n}(\eps,n,k,L,\px) \geq \Cpoinceps
	\end{equation}
\end{theorem}
\begin{proof}
In the support of $\px$, \ie in $\Xset=\R_{\geq0}^n$, $\px$ is greater than zero and constant on the intersection of $\R_{\geq0}^n$ and planes
\begin{equation}
	P_a\defas\set{\x}{x_1+\cdots+x_n=a} = \set{\x}{(1,\ldots,1)\tr\x=a} \subset \R^n, \quad a>0,
\end{equation}
i.e., on hypersurfaces $T_a\defas P_a\cap\R_{\geq0}^n$.
The situation is illustrated by Fig.~\ref{fig:expnd} for $n=3$ dimensions.
\begin{figure}
	\centering
	\includegraphics[width=0.45\linewidth]{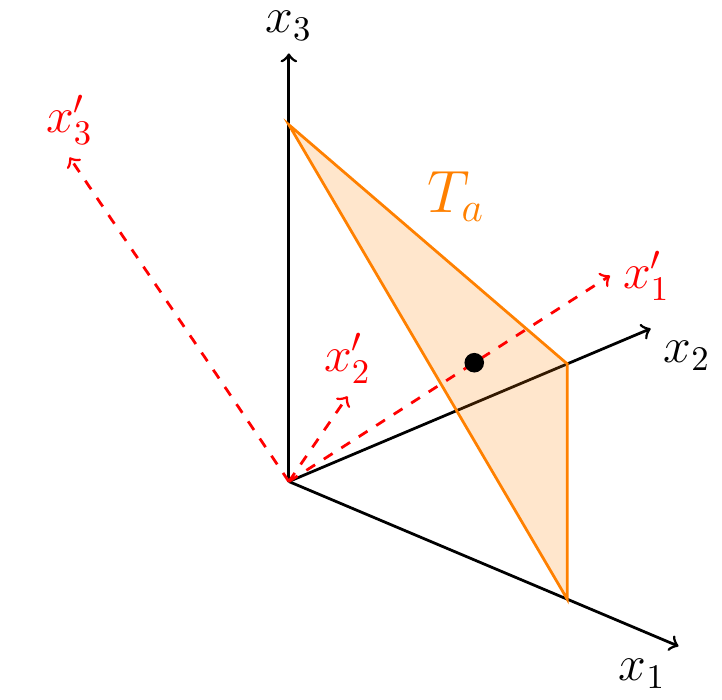}
	\caption{Exponential distribution in $3$D with a rotated coordinate system.}
	\label{fig:expnd}
\end{figure}

For $\x=\xyz{\y,\z}\in\R_{\geq0}^n$, the value of $\pyz(\y,\z)$ is only determined by~$y_1\geq0$.
Reversely, if $y_1<0$, then $\pyz(\y,\z)=0$.
We know that the point at $\x_0\defas\beta(1,\ldots,1)\tr\in\R^n$ with $\norm{\x_0}_2=y_1$ is supposed to lie on $P_a$ for some~$\beta>0$.
It follows immediately that $\beta=y_1/\sqrt{n}$.
Also, we determine~$a$ with
\begin{equation}
	a = (1,\ldots,1)\tr\x_0 = \frac{y_1}{\sqrt{n}}n = \sqrt{n}y_1.
\end{equation}
Let us define $T(y_1)\defas T_{\sqrt{n}y_1}$.
That is,
\begin{equation}
	\pyz(\y,\z) = \exp(-\sqrt{n}y_1) \cdot \ind_{[0,\infty)}(y_1) \cdot \ind_{T(y_1)}(\y,\z).
\end{equation}
$T(y_1)$, as a geometric figure, is a \textit{regular $(n-1)$-simplex} in $n$ dimensions.
$T(y_1)$ is intrinsically $(n-1)$-dimensional and has $n$~corners which are
\begin{equation}
	(\sqrt{n}y_1,0,\ldots,0),\ldots,(0,\ldots,0,\sqrt{n}y_1)\in\R^n.
\end{equation}
It follows that the side length of $T(y_1)$ is $\sqrt{2n}y_1$.
Note that the coordinates $\ytl=(y_2,\ldots,y_k)\tr$ and $\z=(z_1,\ldots,z_{n-k})\tr$ all move on $T(y_1)$.

We can rewrite $T(y_1)$ as
\begin{align}
	T(y_1) &= \set{\x\in\R_{\geq0}^n}{(W\tr\x)_1=y_1} \\
	&= \set{\xyz{\tilde{\y},\tilde{\z}}}{\xyz{\tilde{\y},\tilde{\z}}\in\R_{\geq0}^n,\, \tilde{y}_1=y_1}.
\end{align}
This motivates to view $T(y_1)$ as an $(n-1)$-dimensional set in the rotated coordinate system, i.e., we define
\begin{equation}
	\check{T}(y_1) \defas \set{(\check{\y},\z)\in\R^{k-1}\times\R^{n-k}}{\xyz{(y_1,\check{\y}),\z}\in T(y_1)} \subset \R^{n-1}.
\end{equation}

We observe that the conditioned random variable~$(\check{\Y},\Z)|Y_1$ is uniformly distributed on the regular $(n-1)$-simplex $\check{T}(y_1)$.
The basic idea to get a bound for $\expct[\var(Z_i|\Y)^{(1+\eps)/\eps}]$ is based on the fact that~$z_i$, moving as the $(k+i-1)$-th coordinate inside $\check{T}(y_1)$, takes values in~$[0,h_i(y_1)]$, where $h_i(y_1)$ is the height of a regular $(k+i-1)$-simplex with side length~$\sqrt{2n}y_1$ and is thus bounded.
In general, the height of a regular $n$-simplex is the distance of a vertex to the circumcentre of its opposite regular $(n-1)$-simplex.
By~\cite[p.~367]{buchholz1992perfect}, it holds that
\begin{equation}
	h_i(y_1) = \sqrt{\frac{n(k+i)}{k+i-1}} y_1.
\end{equation}

We start the computation by noting that
\begin{equation}
	\pyz(\y,\z) = \exp(-\sqrt{n}y_1) \cdot \ind_{[0,\infty)}(y_1) \cdot \ind_{\check{T}(y_1)}(\ytl,\z).
\end{equation}
The marginal distribution of $Z_i|\Y$ is given by
\begin{equation}
	\rho_{Z_i|\Y}(z_i|\y) = \int\cdots\int{\pzy(\z|\y)\d{z_1}\ldots\d{z_{i-1}}\d{z_{i+1}}\ldots\d{z_{n-k}}}.
\end{equation}
and so we get
\begin{align}
	&\rho_{\Y}(\y)\rho_{Z_i|\Y}(z_i|\y) \\
	&\qquad =\int\cdots\int{\pyz(\y,\z)\d{z_1}\ldots\d{z_{i-1}}\d{z_{i+1}}\ldots\d{z_{n-k}}} \\
	&\qquad= \exp(-\sqrt{n}y_1) \cdot \ind_{[0,\infty)}(y_1) \\
	&\qquad\qquad\cdot \int\cdots\int{\ind_{T(y_1)}(\ytl,\z)\d{z_1}\ldots\d{z_{i-1}}\d{z_{i+1}}\ldots\d{z_{n-k}}}.
\end{align}
Using Jensen's inequality in a first step, we can continue with
\begin{align}
	&\expct[\var(Z_i|\Y)^{(1+\eps)/\eps}] \leq \expct[\expct[Z_i^{2(1+\eps)/\eps}\|\Y]] \\
	&\qquad= \int \left( \int z_i^{2(1+\eps)/\eps} \, \rho_{Z_i|\Y}(z_i|\y) \d{z_i} \right) \rho_{\y}(\y) \dy \\
	&\qquad= \int_0^\infty{\exp(-\sqrt{n}y_1) \left( \int\int z_i^{2(1+\eps)/\eps}\cdot\ind_{\check{T}(y_1)}(\ytl,\z) \dz\d{\ytl}\right) \d{y_1}} \\
	&\qquad\leq \int_0^\infty{\exp(-\sqrt{n}y_1) \; h_i(y_1)^{2(1+\eps)/\eps} \left( \int\int \ind_{\check{T}(y_1)}(\ytl,\z) \dz\d{\ytl}\right) \d{y_1}} \\
	&\qquad= \int_0^\infty{\exp(-\sqrt{n}y_1) \; \left(\sqrt{\frac{n(k+i)}{k+i-1}} y_1\right)^{2(1+\eps)/\eps} \frac{\sqrt{n^n}}{(n-1)!} y_1^{n-1}} \d{y_1}  \\
	&\qquad= \left(\frac{n(k+i)}{k+i-1}\right)^{(1+\eps)/\eps} \frac{\sqrt{n^n}}{(n-1)!} \int_0^\infty{y_1^{n+1+2/\eps} \exp(-\sqrt{n}y_1) \d{y_1}} \\
	&\qquad= \left(\frac{n(k+i)}{k+i-1}\right)^{(1+\eps)/\eps} \frac{\sqrt{n^n}}{(n-1)!} \frac{\Gamma(n+2+2/\eps)}{n^{(1+\eps)/\eps}\sqrt{n^n}} \\
	&\qquad= \left(\frac{k+i}{k+i-1}\right)^{(1+\eps)/\eps} \frac{\Gamma(n+2+2/\eps)}{(n-1)!}.
\end{align}
Note that an intermediate step of the previous calculation uses the fact that the volume of the regular $(n-1)$-simplex $\check{T}(y_1)$ with side length~$\sqrt{2n}y_1$ is (see~\cite[p.~367]{buchholz1992perfect})
\begin{equation}
	\int\int \ind_{\check{T}(y_1)}(\ytl,\z) \dz \d{\ytl} = \frac{\sqrt{n}^n}{(n-1)!} y_1^{n-1}.
\end{equation}

Remember from~\eqref{eq:bound_Cy_eps} and~\eqref{eq:CvarW} that
\begin{equation}
	\expct[C_\Y^{(1+\eps)/\eps}]^{\eps/(1+\eps)} \leq K (n-k)^{1/(1+\eps)} C_{\var}(\eps,n,k)
\end{equation}
with
\begin{align}
	C_\var(\eps,n,k) &= \left(\sum_{i=1}^{n-k}\expct[\var(Z_i|\Y)^{(1+\eps)/\eps}] \right)^{\eps/(1+\eps)} \\
	&\leq \left(\frac{\Gamma(n+2+2/\eps)}{(n-1)!} \; \sum_{i=1}^{n-k}\left(\frac{k+i}{k+i-1}\right)^{(1+\eps)/\eps} \right)^{\eps/(1+\eps)}.
\end{align}
Defining
\begin{equation}
	\label{eq:C_eps}
	C_\eps(n,k) \defas (n-k)^{1/(1+\eps)} \left(\frac{\Gamma(n+2+2/\eps)}{(n-1)!} \; \sum_{i=1}^{n-k}\left(\frac{k+i}{k+i-1}\right)^{(1+\eps)/\eps} \right)^{\eps/(1+\eps)}
\end{equation}
then yields
\begin{equation}
	\expct[C_\Y^{(1+\eps)/\eps}]^{\eps/(1+\eps)} \leq K C_\eps(n,k).
\end{equation}

Combining all bounds, we get that
\begin{equation}
	\label{eq:C_expn}
	\Cpoinceps(\eps,n,k) \leq K \cdot L^{\eps/(1+\eps)} \cdot C_\eps(n,k) \asdef C_{\exp^n}(\eps,n,k,L),
\end{equation}
where $\Cpoinceps(\eps,n,k)$ was defined in~\eqref{eq:Cpoinceps}.
We recall that $n$ denotes the dimension of the problem, $k$ the dimension of the active subspace, $L$ is the upper bound on~$\norm{\grad f}_2^2$, and~$K$ the universal constant from~\eqref{eq:bound_Cy_eps}.

The result follows by Lemma~\ref{lem:CpoincepsW}.
\end{proof}

Fig.~\ref{fig:bound_n_eps} depicts the quantity $C_\eps(n,k=1)$ from~\eqref{eq:C_eps} as a function of~$\eps>0$ for some $n\in\N$ (left plot) and as a function of $n\geq2$ for several $\eps>0$ (right plot).
We set $k=1$ since this gives the maximum value for $C_\eps$ over all~$k\geq1$.
\begin{figure}
	\centering
	\includegraphics[width=\linewidth]{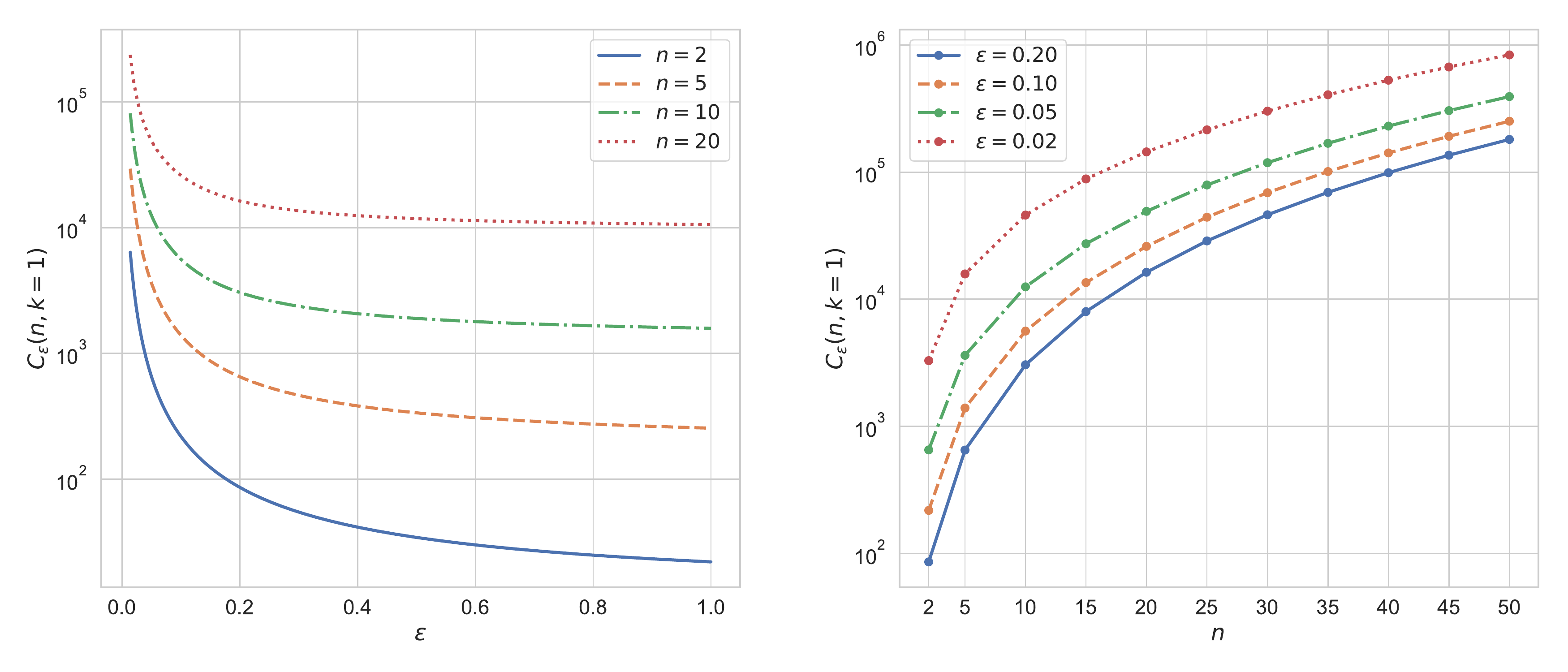}
	\caption{The left plot shows curves of the map $\eps\mapsto C_\eps(n,k=1)$ for $n\in\{ 2,5,10,20 \}$.
		Curves of $n\mapsto C_\eps(n,k=1)$ for $\eps\in\{ 0.02,0.05,0.1,0.2 \}$ are displayed on the right.}
	\label{fig:bound_n_eps}
\end{figure}
As expected, the curves increase quickly as $\eps$ approaches zero or, respectively,~$n$ becomes large.

\begin{remark}
	In the previous theorem, the exponentially distributed random variables are assumed to have unit rates.
	The computations can also be made for arbitrary rates $\nu_i$, $i=1,\ldots,n$.
	However, some modifications are necessary.
	Let $\vec{\nu}=(\nu_1,\ldots,\nu_n)\tr$ denote the vector of rates.
	To get again the worst case scenario as in the previous subsection (uniform distribution on a simplex structure), the coordinate system has to be rotated in such a way that the vector $(1,0,\ldots,0)\tr$ rotates to $\vec{\nu}/\norm{\vec{\nu}}_2$.
	The structure of a regular simplex that is used in the estimates above is not present in this more general case.
	Instead, we get a general simplex whose heights are not as easy to compute as in the regular case.
	However, rough estimates can be achieved by enclosing the general simplex with a larger regular one.
\end{remark}

\section{Future work with MGH distributions}
\label{sec:futu_work}
The generalized bound from Lemma~\ref{lem:CpoincepsW} and the study of corresponding Poincaré type constants~$\CpoincepsW$ and~$\Cpoinceps$ for independently exponentially distributed random variables in Subsection~\ref{ssec:exp_dist_case} motivate further similar investigations of more general distributions.
From a statistical perspective, a study of the class of \textit{multivariate generalized hyperbolic distributions} (MGH) (see \eg~\cite{barndorff1977exponentially}) can be considered as a next step since it allows for distributions with both non-zero skewness and heavier tails.
An MGH is a distribution of the random vector
\begin{equation}
	\label{eq:mgh}
	\X = \mu + \beta A + \sqrt{A}M\V
\end{equation}
with location parameter~$\mu\in\R^n$, skewness parameter~$\beta\in\R^n$, and a symmetric positive definite matrix~$M\in\R^{n\times n}$.
The scalar random variable~$A$, called the mixing variable, follows a generalized inverse Gaussian distribution (GIG)~\cite{jorgensen2012statistical}, and~$\V\sim\mathcal{N}(0,I)$ is independent of~$A$.
As a particular example, for~$\X$ to be Laplace distributed, we set~$\beta=0$ and let~$A$ be exponentially distributed~\cite{kotz2012laplace}.
Note that, however, the example from Subsection~\ref{ssec:exp_dist_case}, assuming independently exponentially distributed random variables, is not an MGH.
In order to include this case, we would need to introduce a mixing random matrix as scaling for~$\V$.

Nevertheless, MGH is a large class containing classical distributions like the normal-inverse Gaussian, generalized Laplace, and Student's t-distribution.
In particular, these distributions are interesting since they have been used in areas like, for instance, economics and financial markets~\cite{barndorff1997normal,barndorff1997processes,eberlein2001application}, spatial and Geostatistics~\cite{bolin2014spatial,bolin2019multivariate,wallin2015geostatistical}, and linear mixed-effects~\cite{asar2018linear,lin2007bayesian,zhang2009robust} which are used, \eg for linear non-Gaussian time series models in medical longitudinal studies~\cite{asar2018linear}.

We mention that, under an assumption on a parameter, MGH distributions are log-concave \cite{yu2017normal}, \ie we can use the estimates on Poincaré constants~$C_\Y$ of Bobkov from~\eqref{eq:C_y_bobkov}.

In our opinion, it is preferable to start the investigation with the subclass of symmetric MGH distributions, \ie~$\beta=0$ in~\eqref{eq:mgh}.
The following lines demonstrate particular difficulties that we already encounter in this smaller subclass.
Let us choose~$\mu=\beta=0$ and~$M=I$ in~\eqref{eq:mgh} such that
\begin{equation}
	\X = \sqrt{A}\V
\end{equation}
with~$\V\sim\mathcal{N}(0,I)$.
A common first step is to study~$\X$ conditioned on~$A$, \ie $\X|A\sim\mathcal{N}(0,AI)$, and to use the tower property of conditional expectations.
That is, analogously to~\eqref{eq:C}, we define
\begin{equation}
	C \defas \expct[C_A] = W\Lambda W\tr
\end{equation}
with
\begin{equation}
	C_A \defas \expctcond{\grad f(\X)\grad f(\X)\tr}{A} = W_A\Lambda_AW_A\tr.
\end{equation}
Choosing~$k\leq n-1$ independent of~$A$, we further set
\begin{equation}
	\Y_A \defas W_{A,1}\tr\X \quad\text{and}\quad \Z_A \defas W_{A,2}\tr\X.
\end{equation}
The computation starts, similar to~\eqref{eq:mse_f_fg_fX_fYZ}, with
\begin{align}
	&\expctcond{(f(\X)-f_g(\X))^2}{A} \\
	&\qquad= \expctcond{\expctcond{(f(\xyzW{\Y_A,\Z_A}{W_A})-g(\Y_A))^2}{\Y_A}}{A} \\
	&\qquad\leq \expctcond{C_{\Y_A} \expctcond{\norm{\grad^{\z_A} f(\xyzW{\Y_A,\Z_A}{W_A})}_2^2}{\Y_A}}{A} \\
	&\qquad= A \; \expctcond{\expctcond{\norm{\grad^{\z_A} f(\xyzW{\Y_A,\Z_A}{W_A})}_2^2}{\Y_A}}{A} \label{eq:C_YA_mgh} \\
	&\qquad= A \; \trace{\Lambda_{A,2}}. \label{eq:trace_L2_mgh}
\end{align}
In~\eqref{eq:C_YA_mgh}, we use the fact that the Poincaré constant of a normal distribution~$\mathcal{N}(0,AI)$ is~$\lambda_{\textnormal{max}}(AI)=A$; see Section~\ref{sec:comp_norm_dists}.
The last step to~\eqref{eq:trace_L2_mgh} is equal to~\eqref{eq:grad_eigvals}.
This yields
\begin{align}
	\expct[(f(\X)-f_g(\X))^2] &= \expct[\expctcond{(f(\X)-f_g(\X))^2}{A}] \\
	&\leq \expct[A\cdot\trace{\Lambda_{A,2}}],
\end{align}
where~the random variable~$A\cdot\trace{\Lambda_{A,2}}$ is assumed to have finite first moment.

At this point, as long as~$A$ is not compactly supported, we can only continue by applying another Hölder's inequality similar to the proof of Lemma~\ref{lem:CpoincepsW}.
However, in any case, we have to face the problem that~$\expct[\trace{\Lambda_{A,2}}]$ is, in general, not equal to~$\trace{\Lambda_2}$ which denotes the inactive trace of~$C$ that we actually aim for.
Nevertheless, we know that
\begin{equation}
	\expct[\trace{\Lambda_A}] = \trace{\Lambda},
\end{equation}
but it is unclear whether, and how, this equality can be exploited for our purposes.

\section{Summary}
\label{sec:summ}
This manuscript discusses bounds for the mean squared error of a given function of interest and a low-dimensional approximation of it which is found by the active subspace method.
These bounds, consisting of the product of a Poincaré constant and a sum of eigenvalues belonging to a non-dominant subspace, are based on a probabilistic Poincaré inequality.
Existing literature applies this Poincaré inequality with indirect non-explicit assumptions that, as a consequence, limit the class of distributions applicable for the active subspace method.
For example, these assumptions exclude distributions with exponential tails as, \eg exponential distributions.
In this respect, the main results of this manuscript give details on the problem that arises when applying the active subspace method with log-concave distributions (which include exponential distributions).
We are able to provide a scenario, involving independently exponentially distributed random variables, in which the usual estimates are not achievable due to an unbounded Poincaré constant.
However, using Hölder's inequality with conjugates $(p,q)$ ($p,q\in(1,\infty)$) instead of $(\infty,1)$,  we show that it is possible to derive a generalized result in a way that enables to balance the size of the Poincaré constant and the remaining order of the error.
We exemplify this trade-off on the mentioned scenario and show that the size of the involved constant is very much depending on the dimension of the problem.
Finally, we propose directions for future work related to the applicability of active subspaces to the large class of multivariate generalized hyperbolic distributions.
Also, details are provided for particular difficulties that already arise with a smaller subclass of these.

\section*{Source code}
Wolfram Mathematica notebooks and code for generating the plots in this manuscript are available in a repository at
\begin{center}
	\url{https://bitbucket.org/m-parente/asm-poincare-pub/}.
\end{center}

\section*{Acknowledgments}
The authors would like to thank Olivier Zahm (INRIA) for pointing out the incorrect application of the Poincaré inequality in~\cite{constantine2014active}.
Also, the helpful and supportive discussions with Krzysztof Podgórski (Lund University) are very much appreciated.

\bibliographystyle{abbrv}

\end{document}